\documentclass[12pt]{amsart}
\usepackage{amsmath,amssymb,amsfonts,amsthm,amscd,indentfirst}
\usepackage{amsmath,amssymb,amsfonts,amsthm,amscd}
\usepackage{indentfirst}
\usepackage{hyperref}

\numberwithin{equation}{section}

\newtheorem{prop}{Proposition}[section]
\newtheorem{theo}[prop]{Theorem}
\newtheorem{lemm}[prop]{Lemma}

\newtheorem{rema}[prop]{Remark}

\def\and{\quad{\rm and}\quad}

\def\<{\langle}
\def\>{\rangle}

\usepackage{graphicx}

\begin{document}
\title[Curvature estimates for Hessian equation in hyperbolic space]{Curvature estimates for semi-convex solutions of Hessian equations in hyperbolic space}
\author[Siyuan Lu]{Siyuan Lu}
\address{Department of Mathematics and Statistics, McMaster University, 
1280 Main Street West, Hamilton, ON, L8S 4K1, Canada.}
\email{siyuan.lu@mcmaster.ca}
\thanks{Research of the author was supported in part by NSERC Discovery Grant.}

\begin{abstract}
In this paper, we establish a curvature estimate for semi-convex solutions of Hessian equations in hyperbolic space. We also obtain a curvature estimate for admissible solutions to prescribed curvature measure type problem in hyperbolic space. A crucial ingredient in both estimates is a concavity inequality for Hessian operator. 
\end{abstract}

\maketitle 

\section{Introduction}

In this paper, we consider the curvature estimates for the following Hessian equation
\begin{align}\label{Equation-Sigma_k}
\sigma_k(\kappa(X))=f(X,\nu(X)),\quad \forall X\in M,
\end{align}
where $\sigma_k$ is the $k$-th elementary symmetric function,  $\nu$ and $\kappa=(\kappa_1,\cdots,\kappa_n)$ are the unit outer normal and principal curvatures of the hypersurface $M$ respectively. $\sigma_k(\kappa)$, $1\leq k\leq n$, are the Weingarten curvatures of $M$. For $k=1,2$ and $n$, they are the mean curvature, scalar curvature and Gauss curvature respectively.

Equation (\ref{Equation-Sigma_k}) arises naturally from many important geometric problems: the Minkowski problem \cite{CY,Nirenberg, Pog53,Pogb}, the prescribed Weingarten curvature problem by Alexandrov \cite{A3,GG}, the prescribed curvature measure problem in convex geometry \cite{A2,GLL,GLM,Pog53} as well as the prescribed curvature problem \cite{BK,CNS86,TW}.

Equation (\ref{Equation-Sigma_k}) in $\mathbb{R}^{n+1}$ has been studied extensively. For $k=1$, it is quasi-linear, curvature estimate follows from the classical theory of quasi-linear PDEs. For $k=n$, it is Monge-Amp\`ere type, curvature estimate was established by Caffarelli, Nirenberg and Spruck \cite{CNS84}. If $f$ is independent of $\nu$, curvature estimate was proved by Caffarelli, Nirenberg and Spruck \cite{CNS85} for a general class of fully nonlinear operators including $\sigma_k$ and $\frac{\sigma_k}{\sigma_l}$. If $f$ only depends on $\nu$, curvature estimate was obtained by Guan and Guan \cite{GG}. Under extra conditions on the dependence of $f$ on $\nu$, $C^2$ estimate for Dirichlet problem of equation (\ref{Equation-Sigma_k}) was established by Ivochkina \cite{I1,I2}. For prescribed curvature measure problem $f(X,\nu)=\left\langle X,\nu\right\rangle\varphi(X)$, curvature estimate was obtained by Guan, Lin and Ma \cite{GLM} and Guan, Li and Li \cite{GLL}.

Equation (\ref{Equation-Sigma_k}) in $\mathbb{R}^{n+1}$ for general $f(X,\nu)$ is a longstanding problem. A breakthrough was made by Guan, Ren and Wang \cite{GRW}. For $k=2$, they obtained a curvature estimate for admissible solutions; for general $k$, they established a curvature estimate for convex solutions, see a simpler proof by Chu \cite{Chu}. Subsequently, Spruck and Xiao \cite{SX} found an elegant proof for $k=2$, which works in $\mathbb{H}^{n+1}$ as well. For $k=n-1,n-2$, curvature estimate was recently proved by Ren and Wang \cite{RW1,RW2}. 

Equation (\ref{Equation-Sigma_k}) in $\mathbb{H}^{n+1}$ is much less studied. Curvature estimate in $\mathbb{H}^{n+1}$ is much more difficult due to the extra negative terms from the interchanging formula. If $f$ is independent of $\nu$, curvature estimate was proved by Jin and Li \cite{JL}. For prescribed curvature measure problem $f(X,\nu)=\left\langle V,\nu\right\rangle\varphi(X)$, curvature estimate was established recently by Yang \cite{Yang}. For general $f(X,\nu)$, curvature estimate was proved by Spruck and Xiao \cite{SX} for $k=2$ mentioned above. In \cite{CLW}, Chen, Li and Wang obtained a curvature estimate for $k=n-1$; they also proved a curvature estimate for convex solutions for general $k$.

\medskip

From analysis point of view, it is desirable to weaken the convexity assumption in \cite{GRW}. In this paper, we establish a curvature estimate under semi-convex assumption. 

We say a hypersurface $M$ is semi-convex if there exists a constant $K>0$ such that
\begin{align*}
\kappa_i(X)\geq -K,\quad 1\leq i\leq n, \quad \forall X\in M.
\end{align*}

Denote Garding's $\Gamma_k$ cone by
\begin{align*}
\Gamma_k=\{\lambda\in \mathbb{R}^n:\sigma_j(\lambda)>0,1\leq j\leq k\}.
\end{align*}

We are now in position to state our main theorem.
\begin{theo}\label{Theorem-General}
Let $M$ be a closed, semi-convex, strictly star-shaped hypersurface satisfying curvature equation (\ref{Equation-Sigma_k}) in $\mathbb{H}^{n+1}$ with $\kappa\in \Gamma_k$. Let $f\in C^2(\Gamma)$ be a positive function, where $\Gamma$ is an open neighbourhood of the unit normal bundle of $M$ in $\mathbb{H}^{n+1}\times \mathbb{S}^n$. Then we have
\begin{align*}
\max_{X\in M;1\leq i\leq n}|\kappa_i(X)|\leq C,
\end{align*}
where $C$ is a constant depending only on $n,k,\|M\|_{C^1}, \inf f$ and $\|f\|_{C^2}$.
\end{theo}

The corresponding result in $\mathbb{R}^{n+1}$ was mentioned in \cite{GRW}, see Remark 4.7 in \cite{GRW}. We emphasize that our method is different from \cite{GRW}. First of all, we have extra terms from the interchanging formula to handle compared to \cite{GRW}. Secondly, our test function is different from the one in \cite{GRW}. As a result, our proof is more straightforward and simpler. Most importantly, the crucial concavity inequality we derived works for all admissible solutions, see Lemma \ref{Concavity-Lemma} below. 

In fact, our concavity inequality is very robust, it enables us to establish a curvature estimate for admissible solutions to the following prescribed curvature measure type problem
\begin{align}\label{Equation-Curvature-Measure}
\sigma_k(\kappa(X))=u^p(X)\varphi(X),\quad \forall X\in M,
\end{align}
where $u=\left\langle V,\nu\right\rangle$ is the support function of $M$ in $\mathbb{H}^{n+1}$, see precise definition in Section 2 below.

\begin{theo}\label{Theorem-Curvature-Measure}
Let $M$ be a closed, strictly star-shaped hypersurface satisfying curvature equation (\ref{Equation-Curvature-Measure}) in $\mathbb{H}^{n+1}$ with $\kappa\in \Gamma_k$. Let $p\in (-\infty,0)\cup (0,1] $ and let $\varphi\in C^2(M)$ be a positive function. Then we have
\begin{align*}
\max_{X\in M;1\leq i\leq n}|\kappa_i(X)|\leq C,
\end{align*}
where $C$ is a constant depending only on $n,k,p,\|M\|_{C^1}, \inf \varphi$ and $\|\varphi\|_{C^2}$.
\end{theo}

The corresponding result in $\mathbb{R}^{n+1}$ was proved by Huang and Xu \cite{HX}, see also \cite{ChenC}. For $p=1$, it is the prescribed curvature measure problem in $\mathbb{H}^{n+1}$ and its curvature estimate was recently obtained by Yang \cite{Yang}.

\medskip

The major difficulty to establish curvature estimates for equation (\ref{Equation-Sigma_k}) is to handle third order terms. Inspired by the work of Brendle, Choi and Daskalopoulos \cite{BCD}, our test function utilizes the largest principle curvature, which gives more good third order terms. In case of multiplicity of largest principle curvature, we are able to treat it in the viscosity sense. As a result, our proof is more straightforward and simpler. The crucial concavity inequality enables us to handle the term involving $h_{111}$, which is the most difficult term. We believe the concavity inequality can be used in other settings besides Theorem \ref{Theorem-Curvature-Measure}. In case that the concavity inequality does not apply, we are able to immediately obtain the estimate by the equation and the semi-convexity assumption. Here the semi-convexity assumption is crucially used. We remark that the curvature estimate for admissible solutions of equation (\ref{Equation-Sigma_k}) remains open for $2<k<n-2$.

\medskip

The organization of the paper is as follows. In Section 2, we collect some formulas and lemmas for Hessian operator and the geometry of hypersurfaces. In Section 3, we establish the crucial concavity inequality. We will prove Theorem \ref{Theorem-General} and Theorem \ref{Theorem-Curvature-Measure} in Section 4 and Section 5 respectively.

\section{Preliminaries}

In this section, we will collect some basic formulas and lemmas for Hessian operator as well as hypersurfaces in $\mathbb{H}^{n+1}$.

\medskip

Let $\lambda=(\lambda_1,\cdots,\lambda_n)\in \mathbb{R}^n$, we will denote 
\begin{align*}
(\lambda|i)=(\lambda_1,\cdots,\lambda_{i-1},\lambda_{i+1},\cdots,\lambda_n)\in \mathbb{R}^{n-1},
\end{align*}
i.e. $(\lambda|i)$ is the vector obtained by deleting the $i$-th component of the vector $\lambda$. Similarly, $(\lambda|ij)$ is the vector obtained by deleting the $i$-th and $j$-th components of the vector $\lambda$.

We now collect some basic properties of Hessian operator, see for instance in \cite{RW2}.
\begin{lemm}\label{Sigma_k-Lemma-0}
For any $\lambda=(\lambda_1,\cdots,\lambda_n)\in \mathbb{R}^n$, we have
\begin{align*}
\sigma_k(\lambda)=\lambda_i\sigma_{k-1}(\lambda|i)+\sigma_k(\lambda|i),\quad \sum_i\sigma_k(\lambda|i)=(n-k)\sigma_{k-1}(\lambda),\quad \sum_i\lambda_i\sigma_{k-1}(\lambda_i)=k\sigma_k(\lambda).
\end{align*}
\end{lemm}

\begin{lemm}\label{Sigma_k-lemma}
Let $\lambda=(\lambda_1,\cdots,\lambda_n)\in \Gamma_k$ with $\lambda_1\geq\cdots\geq\lambda_n$. 
\begin{enumerate}
\item If $\lambda_i\leq 0$, then we have
\begin{align*}
-\lambda_i\leq \frac{(n-k)}{k}\lambda_1.
\end{align*}
\item For any $1\leq l<k$, we have
\begin{align*}
\sigma_l(\lambda)\geq \lambda_1\cdots\lambda_l.
\end{align*}
\item \begin{align*}
\lambda_1\sigma_{k-1}(\lambda|1)\geq C(n,k)\sigma_k(\lambda),
\end{align*}
where $C(n,k)>0$ is a constant depending only on $n$ and $k$.
\item 
\begin{align*}
\sum_i\lambda_i^2\sigma_{k-1}(\lambda|i)\geq \frac{k}{n}\sigma_1(\lambda)\sigma_k(\lambda).
\end{align*}
\end{enumerate}
\end{lemm}

Let $\lambda(A)$ be the eigenvalue vector of a symmetric matrix $A=(a_{ij})$. Then we can define a function $F$ on the set of symmetric matrices by
\begin{align*}
F(A)=f(\lambda(A)).
\end{align*}
Denote
\begin{align*}
F^{pq}=\frac{\partial F}{\partial a_{pq}},\quad F^{pq,rs}=\frac{\partial^2F}{\partial a_{pq}\partial a_{rs}}.
\end{align*}

Suppose $A$ is diagonalized at $x_0$, then at $x_0$, we have
\begin{align*}
\sigma_k^{pq}(A)=\frac{\partial \sigma_k}{\partial \lambda_p}(\lambda)\delta_{pq}=\sigma_{k-1}(\lambda|p)\delta_{pq},
\end{align*}
\begin{align*}
\sigma_k^{pq,rs}(A)=\begin{cases}
\frac{\partial^2 \sigma_k}{\partial \lambda_p\partial\lambda_r}(\lambda)=\sigma_{k-2}(\lambda|pr),\quad &p=q,r=s,p\neq r,\\
-\frac{\partial^2 \sigma_k}{\partial \lambda_p\partial\lambda_q}(\lambda)=-\sigma_{k-2}(\lambda|pq), &p=s,q=r, p\neq q,\\
0, &\textit{otherwise}.
\end{cases}
\end{align*}

We also need the following lemma, which is essentially contained in \cite{GRW}. 
\begin{lemm}\label{GLL-Lemma}
Let $\lambda=(\lambda_1,\cdots,\lambda_n)\in \Gamma_k$ and let $1\leq l<k$, then we have
\begin{align*}
-\sum_{p\neq q}\frac{\sigma_k^{pp,qq}(\lambda)\xi_p\xi_q}{\sigma_k}+\sum_{p\neq q}\frac{\sigma_l^{pp,qq}(\lambda)\xi_p\xi_q}{\sigma_l}\geq-\frac{\left( \sum_i \sigma_k^{ii}(\lambda)\xi_i \right)^2}{\sigma_k^2}+\frac{\left( \sum_i \sigma_l^{ii}(\lambda)\xi_i \right)^2}{\sigma_l^2},
\end{align*}
where $\xi=(\xi_1,\cdots,\xi_n)$ is an arbitrary vector in $\mathbb{R}^n$.
\end{lemm}

\begin{proof}
Let $\alpha=\frac{1}{k-l}$, by Lemma 2.2 in \cite{GRW}, we have
\begin{align*}
-\frac{\sigma_k^{pp,qq}}{\sigma_k}+\frac{\sigma_l^{pp,qq}}{\sigma_l}\geq \left(\frac{\sigma_k^{pp}}{\sigma_k}-\frac{\sigma_l^{pp}}{\sigma_l}\right)\left((\alpha-1)\frac{\sigma_k^{qq}}{\sigma_k}-(\alpha+1)\frac{\sigma_l^{qq}}{\sigma_l}\right)
\end{align*}
in the sense of comparison of symmetric matrices.

Contracting with $\xi$, we have
\begin{align*}
&\ -\sum_{p\neq q}\frac{\sigma_k^{pp,qq}\xi_p\xi_q}{\sigma_k}+\sum_{p\neq q}\frac{\sigma_l^{pp,qq}\xi_p\xi_q}{\sigma_l}\\
\geq &\ \left(\frac{\sum_p\sigma_k^{pp}\xi_p}{\sigma_k}-\frac{\sum_p\sigma_l^{pp}\xi_p}{\sigma_l}\right)\left((\alpha-1)\frac{\sum_q\sigma_k^{qq}\xi_q}{\sigma_k}-(\alpha+1)\frac{\sum_q\sigma_l^{qq}\xi_q}{\sigma_l}\right)\\
=&\ (\alpha-1) \frac{\left( \sum_i \sigma_k^{ii}\xi_i \right)^2}{\sigma_k^2}-2\alpha \frac{ \sum_i \sigma_k^{ii}\xi_i }{\sigma_k}\frac{ \sum_i \sigma_l^{ii}\xi_i }{\sigma_l}+(\alpha+1) \frac{\left( \sum_i \sigma_l^{ii}\xi_i \right)^2}{\sigma_l^2}\\
\geq &\ -\frac{\left( \sum_i \sigma_k^{ii}\xi_i \right)^2}{\sigma_k^2}+ \frac{\left( \sum_i \sigma_l^{ii}\xi_i \right)^2}{\sigma_l^2}.
\end{align*}

The lemma is now proved.
\end{proof}

Let $\mathbb{H}^{n+1}$ be the hyperbolic space, then the metric can be written as
\begin{align*}
ds^2=dr^2+\phi^2(r)d\sigma^2,
\end{align*}
where $d\sigma^2$ is the standard metric of $\mathbb{S}^n$ and $\phi(r)=\sinh(r)$.

Define 
\begin{align*}
\Phi(r)=\int_0^r\phi(\rho)d\rho,\quad V=\phi(r)\frac{\partial}{\partial r}.
\end{align*}

We state some well-known lemmas, see for instance in \cite{LuCAG}.
\begin{lemm}\label{Phi}
Let $M$ be a hypersurface in $\mathbb{H}^{n+1}$ with induced metric $g$, then $\Phi|_M$ satisfies
\begin{align*}
\Phi_{ij} =\phi^\prime(r)g_{ij}-h_{ij}\left\langle V,\nu\right\rangle,
\end{align*}
where $\nu$ and $h_{ij}$ are the unit outer normal and the second fundamental form of $M$ respectively.
\end{lemm}

\begin{lemm}\label{support function}
Let $u=\left\langle V,\nu\right\rangle$ be the support function, then we have
\begin{align*}
u_i =&\ \sum_kg^{kl}h_{ik}\Phi_l,\\
u_{ij}=&\ \sum_{k,l}g^{kl}h_{ijk}\Phi_l+\phi^\prime h_{ij}-\sum_{k,l}g^{kl}h_{ik}h_{jl}u,
\end{align*}
where $(g^{ij})$ is the inverse matrix of $(g_{ij})$.
\end{lemm}

For a fixed local frame $(e_1,\cdots,e_n)$, the Gauss and Codazzi equations are given by
\begin{align*}
R_{ijkl}=-(\delta_{ik}\delta_{jl}-\delta_{il}\delta_{jk})+\left(h_{ik}h_{jl}-h_{il}h_{jk}\right),
\end{align*}
\begin{align*}
h_{ijk}=h_{ikj}.
\end{align*}
The convention that $R_{ijij}$ denotes the sectional curvature is used here.

The interchanging formula is given by
\begin{align}\label{comm}
h_{klij}=&\ h_{ijkl}-\sum_m h_{ml}(h_{im}h_{kj}-h_{ij}h_{mk})-\sum_m h_{mj}(h_{mi}h_{kl}-h_{il}h_{mk})\\ \nonumber
&\ -\sum_m h_{ml}(\delta_{ij}\delta_{km}-\delta_{ik}\delta_{jm}) -\sum_m h_{mj}(\delta_{il}\delta_{km}-\delta_{ik}\delta_{lm}).
\end{align}

\section{A concavity inequality}

In this section, we will prove the following concavity inequality for Hessian operator, which is the key to establish curvature estimates. The lemma is inspired by \cite{GRW} and \cite{Yang}.
\begin{lemm}\label{Concavity-Lemma}
Let $\lambda=(\lambda_1,\cdots,\lambda_n)\in \Gamma_k$ with $\lambda_1\geq\cdots\geq\lambda_n$ and let $1\leq l<k$. For any $\epsilon,\delta,\delta_0\in (0,1)$, there exists a constant $\delta^\prime>0$ depending only on $\epsilon,\delta,\delta_0,n,k$ and $l$ such that if $\lambda_l\geq \delta \lambda_1$ and $\lambda_{l+1}\leq \delta^\prime\lambda_1$, then we have
\begin{align*}
-\sum_{p\neq q}\frac{\sigma_k^{pp,qq}\xi_p\xi_q}{\sigma_k}+\frac{\left( \sum_i \sigma_k^{ii}\xi_i \right)^2}{\sigma_k^2}\geq (1-\epsilon)\frac{\xi_1^2}{\lambda_1^2}-\delta_0\sum_{i>l}\frac{\sigma_k^{ii}\xi_i^2}{\lambda_1 \sigma_k},
\end{align*}
where $\xi=(\xi_1,\cdots,\xi_n)$ is an arbitrary vector in $\mathbb{R}^n$.
\end{lemm}

\begin{proof}
By Lemma \ref{GLL-Lemma}, we have
\begin{align}\label{Concavity-1}
&\ -\sum_{p\neq q}\frac{\sigma_k^{pp,qq}\xi_p\xi_q}{\sigma_k}+\frac{\left( \sum_i \sigma_k^{ii}\xi_i \right)^2}{\sigma_k^2}\\\nonumber
\geq&\  \frac{\left( \sum_i \sigma_l^{ii}\xi_i \right)^2}{\sigma_l^2}-\sum_{p\neq q}\frac{\sigma_l^{pp,qq}\xi_p\xi_q}{\sigma_l}\\\nonumber
=&\ \frac{1}{\sigma_l^2}\left(\sum_i \left(\sigma_l^{ii}\xi_i \right)^2+ \sum_{p\neq q}\left(\sigma_l^{pp}\sigma_l^{qq}-\sigma_l\sigma_l^{pp,qq}\right) \xi_p\xi_q\right).
\end{align}

We claim that
\begin{align}\label{Concavity-2}
\sum_{p\neq q}\left(\sigma_l^{pp}\sigma_l^{qq}-\sigma_l\sigma_l^{pp,qq}\right) \xi_p\xi_q\geq -\frac{\epsilon}{2} \sum_{i\leq l}\left(\sigma_l^{ii}\xi_i\right)^2-\frac{C}{\epsilon}\sum_{i> l}\left(\sigma_l^{ii}\xi_i\right)^2,
\end{align}
where $C$ is a constant depending only on $n,k$ and $l$.

For $l=1$, we have $\sigma_l^{ii}=1$ and $\sigma_l^{pp,qq}=0$, thus
\begin{align*}
&\ \sum_{p\neq q}\left(\sigma_l^{pp}\sigma_l^{qq}-\sigma_l\sigma_l^{pp,qq}\right) \xi_p\xi_q\\
=&\ \sum_{p\neq q}\xi_p\xi_q\geq  2\sum_{i\neq 1}\xi_1\xi_i-C\sum_{i\neq 1}\xi_i^2\\
\geq&\ -\frac{\epsilon}{2}\xi_1^2-\frac{C}{\epsilon}\sum_{i\neq 1}\xi_i^2.
\end{align*}
Thus the claim holds for $l=1$.

For $l\geq 2$, by (4.22) in \cite{GRW}, we have
\begin{align}\label{Concavity-2-0}
 \sigma_l^{pp}\sigma_l^{qq}-\sigma_l\sigma_l^{pp,qq}= \sigma_{l-1}^2(\lambda|pq)-\sigma_{l}(\lambda|pq)\sigma_{l-2}(\lambda|pq)\geq 0.
\end{align}

It follows that
\begin{align}\label{Concavity-2-1}
&\ \sum_{p\neq q;p,q\leq l}\left(\sigma_l^{pp}\sigma_l^{qq}-\sigma_l\sigma_l^{pp,qq}\right) \xi_p\xi_q\\\nonumber
\geq &\ -\sum_{p\neq q;p,q\leq l} \left( \sigma_{l-1}^2(\lambda|pq)-\sigma_{l}(\lambda|pq)\sigma_{l-2}(\lambda|pq)\right) |\xi_p\xi_q|.
\end{align}

By Lemma \ref{Sigma_k-lemma} and the assumption of the lemma, for $p,q\leq l$, we have
\begin{align}\label{Concavity-3}
\sigma_l(\lambda|p)\leq&\  C\frac{\lambda_1\cdots\lambda_{l+1}}{\lambda_p}\leq \frac{C\delta^\prime}{\delta}\lambda_1\cdots\lambda_l\leq \frac{C\delta^\prime}{\delta}\sigma_l,\\\nonumber
\sigma_{l-1}(\lambda|pq) \leq&\  C\frac{\lambda_1\cdots\lambda_{l+1}}{\lambda_p\lambda_q}\leq \frac{C\delta^\prime}{\delta}\frac{\lambda_1\cdots\lambda_l}{\lambda_p}\leq \frac{C\delta^\prime}{\delta} \frac{\sigma_l}{\lambda_p}.
\end{align}

Together with Lemma \ref{Sigma_k-Lemma-0}, we have
\begin{align}\label{Concavity-4}
\lambda_p\sigma_l^{pp}=\sigma_l-\sigma_l(\lambda|p)\geq \left(1-\frac{C\delta^\prime}{\delta}\right)\sigma_l\geq \frac{1}{2}\sigma_l,
\end{align}
by choosing $\delta^\prime$ sufficiently small.

Combining (\ref{Concavity-3}) and (\ref{Concavity-4}), we have
\begin{align*}
\sigma_{l-1}(\lambda|pq)\leq \frac{C\delta^\prime}{\delta} \frac{\sigma_l}{\lambda_p}\leq \frac{C\delta^\prime}{\delta} \sigma_l^{pp}.
\end{align*}

By symmetry, we have
\begin{align*}
\sigma_{l-1}(\kappa|pq)\leq \frac{C\delta^\prime}{\delta} \sigma_l^{qq}.
\end{align*}

Therefore, 
\begin{align}\label{Concavity-5}
\sigma_{l-1}^2(\lambda|pq)\leq C\left(\frac{\delta^\prime}{\delta} \right)^2\sigma_l^{pp}\sigma_l^{qq}.
\end{align}

Similarly, by Lemma \ref{Sigma_k-lemma} and the assumption of the lemma as well as (\ref{Concavity-4}), for $p,q\leq l$, we have
\begin{align}\label{Concavity-6}
\left|\sigma_{l}(\lambda|pq)\sigma_{l-2}(\lambda|pq)\right|\leq &\ C\frac{\lambda_1\cdots\lambda_{l+1} \cdot |\lambda_{l+2}| }{\lambda_p\lambda_q} \cdot \frac{\lambda_1\cdots\lambda_l}{\lambda_p\lambda_q}\\\nonumber
\leq&\  C\left(\frac{\delta^\prime}{\delta}\right)^2 \frac{\lambda_1\cdots\lambda_{l}}{\lambda_p} \cdot \frac{\lambda_1\cdots\lambda_l}{\lambda_q}\\\nonumber
\leq&\ C \left(\frac{\delta^\prime}{\delta}\right)^2 \frac{\sigma_l}{\lambda_p}\cdot\frac{\sigma_l}{\lambda_q} \\\nonumber
\leq &\   C \left(\frac{\delta^\prime}{\delta}\right)^2   \sigma_l^{pp}\sigma_l^{qq}.
\end{align}

Plugging (\ref{Concavity-5}) and (\ref{Concavity-6}) into (\ref{Concavity-2-1}), we have
\begin{align}\label{Concavity-7}
&\ \sum_{p\neq q;p,q\leq l}\left(\sigma_l^{pp}\sigma_l^{qq}-\sigma_l\sigma_l^{pp,qq}\right) \xi_p\xi_q\\\nonumber
\geq&\ -C \left(\frac{\delta^\prime}{\delta}\right)^2  \sum_{p\neq q;p,q\leq l} \sigma_l^{pp}\sigma_l^{qq}|\xi_p\xi_q|\\\nonumber
\geq&\ -\frac{\epsilon}{4}\sum_{i\leq l} \left(\sigma_l^{ii}\xi_i\right)^2,
\end{align}
by choosing $\delta^\prime$ sufficiently small.

By (\ref{Concavity-2-0}), we have
\begin{align}\label{Concavity-8}
&\ 2\sum_{p\leq l,q>l}\left(\sigma_l^{pp}\sigma_l^{qq}-\sigma_l\sigma_l^{pp,qq}\right) \xi_p\xi_q+\sum_{p\neq q;p,q>l}\left(\sigma_l^{pp}\sigma_l^{qq}-\sigma_l\sigma_l^{pp,qq}\right) \xi_p\xi_q\\\nonumber
\geq &\ -2\sum_{p\leq l,q>l}\left(\sigma_l^{pp}\sigma_l^{qq}-\sigma_l\sigma_l^{pp,qq}\right) |\xi_p\xi_q|-\sum_{p\neq q; p,q>l}\left(\sigma_l^{pp}\sigma_l^{qq}-\sigma_l\sigma_l^{pp,qq}\right) |\xi_p\xi_q|\\\nonumber
\geq &\ -2\sum_{p\leq l,q>l}\sigma_l^{pp}\sigma_l^{qq}|\xi_p\xi_q|-\sum_{p\neq q; p,q>l} \sigma_l^{pp}\sigma_l^{qq} |\xi_p\xi_q|\\\nonumber
\geq&\ -\frac{\epsilon}{4} \sum_{i\leq l}\left(\sigma_l^{ii}\xi_i\right)^2-\frac{C}{\epsilon}\sum_{i> l}\left(\sigma_l^{ii}\xi_i\right)^2.
\end{align}

The claim now follows by combining (\ref{Concavity-7}) and (\ref{Concavity-8}).

\medskip

By (\ref{Concavity-1}) and (\ref{Concavity-2}), we have
\begin{align*}
&\ -\sum_{p\neq q}\frac{\sigma_k^{pp,qq}\xi_q\xi_q}{\sigma_k}+\frac{\left(\sum_i \sigma_k^{ii}\xi_i\right) ^2}{\sigma_k^2}\\
\geq&\ \frac{1}{\sigma_l^2}\left(\sum_i \left(\sigma_l^{ii}\xi_i\right)^2-\frac{\epsilon}{2} \sum_{i\leq l}\left(\sigma_l^{ii}\xi_i\right)^2-\frac{C}{\epsilon}\sum_{i> l}\left(\sigma_l^{ii}\xi_i\right)^2\right)\\
\geq &\ \frac{1}{\sigma_l^2}\left(\left( 1-\frac{\epsilon}{2} \right)\sum_{i\leq l}\left(\sigma_l^{ii}\xi_i\right)^2-\frac{C}{\epsilon}\sum_{i> l}\left(\sigma_l^{ii}\xi_i\right)^2\right)\\
\geq &\ \frac{1}{\sigma_l^2}\left(\left( 1-\frac{\epsilon}{2} \right)\left(\sigma_l^{11}\xi_1\right)^2-\frac{C}{\epsilon}\sum_{i>l}\left(\sigma_l^{ii}\xi_i\right)^2\right).
\end{align*}

By (\ref{Concavity-4}), we have
\begin{align*}
\left( 1-\frac{\epsilon}{2} \right)\left(\frac{\sigma_l^{11}}{\sigma_l}\right)^2 \geq \left( 1-\frac{\epsilon}{2} \right)\left(\frac{1-\frac{C\delta^\prime}{\delta}}{\lambda_1}\right)^2\geq \frac{1-\epsilon}{\lambda_1^2},
\end{align*}
by choosing $\delta^\prime$ sufficiently small.

It follows that
\begin{align*}
-\sum_{p\neq q}\frac{\sigma_k^{pp,qq}\xi_q\xi_q}{\sigma_k}+\frac{\left(\sum_i \sigma_k^{ii}\xi_i\right) ^2}{\sigma_k^2}\geq   (1-\epsilon)\frac{\xi_1^2}{\lambda_1^2}-\frac{C}{\epsilon}\frac{1}{\sigma_l^2}\sum_{i> l}\left(\sigma_l^{ii}\xi_i\right)^2.
\end{align*}

By Lemma \ref{Sigma_k-lemma} and the assumption of the lemma, for $i>l$, we have
\begin{align*}
-\frac{C}{\epsilon}\left(\frac{\sigma_l^{ii}}{\sigma_l}\right)^2\geq -\frac{C}{\epsilon}\left(\frac{\lambda_1\cdots\lambda_{l-1}}{\lambda_1\cdots\lambda_l}\right)^2\geq -\frac{C}{\epsilon \delta^2}\frac{1}{\lambda_1^2}.
\end{align*}

Consequently,
\begin{align*}
-\sum_{p\neq q}\frac{\sigma_k^{pp,qq}\xi_q\xi_q}{\sigma_k}+\frac{\left(\sum_i \sigma_k^{ii}\xi_i\right) ^2}{\sigma_k^2}\geq   (1-\epsilon)\frac{\xi_1^2}{\lambda_1^2}-\frac{C}{\epsilon \delta^2}\sum_{i>l}\frac{\xi_i^2}{\lambda_1^2}.
\end{align*}

To prove the lemma, we only need to show for $i>l$, we have
\begin{align*}
-\frac{C}{\epsilon \delta^2} \frac{1}{\lambda_1^2}\geq -\delta_0\frac{\sigma_k^{ii}}{\lambda_1\sigma_k},
\end{align*}
i.e.
\begin{align*}
\sigma_k^{ii}\geq \frac{C}{\epsilon \delta^2\delta_0}\frac{\sigma_k}{\lambda_1}.
\end{align*}

Since $\sigma_l^{nn}\geq \cdots\geq \sigma_k^{11}$, we only need to prove the above inequality for $i=l+1$. From now on, we will fix $i=l+1$.

By Lemma \ref{Sigma_k-Lemma-0} and Lemma \ref{Sigma_k-lemma}, we have
\begin{align*}
\frac{C}{\epsilon \delta^2\delta_0}\frac{\sigma_k}{\lambda_1}\leq \frac{C}{\epsilon \delta^2\delta_0}\sigma_k^{11}=\frac{C}{\epsilon \delta^2\delta_0}\bigg(\lambda_i\sigma_{k-2}(\lambda|1i)+\sigma_{k-1}(\lambda|1i)\bigg).
\end{align*}

If $\sigma_{k-1}(\lambda|1i)\leq 0$, then by the assumption of the lemma, we have
\begin{align}\label{Concavity-9}
\frac{C}{\epsilon \delta^2\delta_0}\frac{\sigma_k}{\lambda_1}\leq \frac{C}{\epsilon \delta^2\delta_0} \lambda_i\sigma_{k-2}(\lambda|1i) \leq \frac{C\delta^\prime} {\epsilon \delta^2\delta_0} \lambda_1\sigma_{k-2}(\lambda|1i).
\end{align}

If $\sigma_{k-1}(\lambda|1i)> 0$, then $(\lambda|1i)\in \Gamma_{k-1}$. By Lemma \ref{Sigma_k-lemma}, we have
\begin{align*}
\sigma_{k-1}(\lambda|1i) \leq&\ C\frac{ \lambda_2\cdots\lambda_{k+1}}{\lambda_i}\\
\leq&\ C\frac{ \lambda_2\cdots\lambda_k}{\lambda_i}\cdot \delta^\prime\lambda_1\\
\leq &\ C\sigma_{k-2}(\lambda|1i)\cdot \delta^\prime\lambda_1.
\end{align*}

Together with the assumption of the lemma, we have
\begin{align}\label{Concavity-10}
\frac{C}{\epsilon \delta^2\delta_0}\frac{\sigma_k}{\lambda_1}\leq \frac{C}{\epsilon \delta^2\delta_0}\bigg( \lambda_i\sigma_{k-2}(\lambda|1i)+\sigma_{k-1}(\lambda|1i)\bigg)\leq\frac{C\delta^\prime}{\epsilon \delta^2\delta_0}\lambda_1\sigma_{k-2}(\lambda|1i).
\end{align}

Combining (\ref{Concavity-9}) and (\ref{Concavity-10}), together with Lemma \ref{Sigma_k-Lemma-0} and the assumption of the lemma, we have
\begin{align*}
\frac{C}{\epsilon \delta^2\delta_0}\frac{\sigma_k}{\lambda_1}\leq&\  \frac{C\delta^\prime}{\epsilon \delta^2\delta_0}\lambda_1\sigma_{k-2}(\lambda|1i)\\
\leq&\  \frac{C\delta^\prime}{\epsilon \delta^2\delta_0}\left(\lambda_1-\lambda_i\right)\sigma_{k-2}(\lambda|1i)\\
=&\ \frac{C\delta^\prime}{\epsilon \delta^2\delta_0}\left(\sigma_k^{ii}-\sigma_k^{11}\right)\\
\leq&\ \frac{C\delta^\prime}{\epsilon \delta^2\delta_0}\sigma_k^{ii}\leq\sigma_k^{ii},
\end{align*}
by choosing $\delta^\prime$ sufficiently small.

The lemma is now proved.

\end{proof}

\section{Curvature equation in general form}

In this section, we will prove Theorem \ref{Theorem-General}. It is a consequence of the following theorem.
\begin{theo}
Let $M$ be a semi-convex, strictly star-shaped hypersurface satisfying curvature equation (\ref{Equation-Sigma_k}) in $\mathbb{H}^{n+1}$ with $\kappa\in \Gamma_k$. Let $f\in C^2(\Gamma)$ be a positive function, where $\Gamma$ is an open neighbourhood of the unit normal bundle of $M$ in $\mathbb{H}^{n+1}\times \mathbb{S}^n$. Then we have
\begin{align*}
\max_{X\in M;1\leq i\leq n}|\kappa_i(X)|\leq C\left(1+\max_{X\in \partial M;1\leq i\leq n}|\kappa_i(X)|\right),
\end{align*}
where $C$ is a constant depending only on $n,k,\|M\|_{C^1}, \inf f$ and $\|f\|_{C^2}$.
\end{theo}

\begin{proof}
Since $k=1,2$ was already solved, we only need to consider the case $k\geq 3$. Since $M$ is semi-convex, strictly star-shaped, without loss of generality, we may assume $\kappa_i\geq -K$ and $u>2a>0$. 

Consider the test function
\begin{align*}
Q=\ln \kappa_1-N\ln u +\alpha\Phi,
\end{align*}
where $\kappa_1$ is the largest principle curvature and $N,\alpha$ are large constants to be determined later.

In the following, we will do a standard but slightly lengthy computation at the maximum point. First time readers can start from (\ref{General-9}).

Assume $Q$ achieves maximum at an interior point $X_0$. At $X_0$, we can choose an orthonormal frame such that $(h_{ij})$ is diagonalized. Without loss of generality, we may assume $\kappa_1$ has multiplicity $m$, i.e.
\begin{align*}
\kappa_1=\cdots=\kappa_m>\kappa_{m+1}\geq \cdots\geq \kappa_n.
\end{align*}

By Lemma 5 in \cite{BCD}, at $X_0$, we have
\begin{align}\label{BCD}
\delta_{kl}\cdot{\kappa_1}_i=h_{kli},\quad 1\leq k,l\leq m,
\end{align}
\begin{align*}
{\kappa_1}_{ii}\geq h_{11ii}+2\sum_{p>m}\frac{h_{1pi}^2}{\kappa_1 -\kappa_p},
\end{align*}
in the viscosity sense.

At $X_0$, we have
\begin{align}\label{General-critical}
0=\frac{{\kappa_1}_i}{\kappa_1}-N\frac{u_i}{u}+\alpha\Phi_i=\frac{h_{11i}}{\kappa_1}-N\frac{u_i}{u}+\alpha\Phi_i,
\end{align}
\begin{align}\label{General-1}
0\geq &\ \frac{{\kappa_1}_{ii}}{\kappa_1}-\frac{ (\kappa_1)_i^2}{\kappa_1^2}-N\frac{u_{ii}}{u}+N\frac{u_i^2}{u^2}+\alpha\Phi_{ii}\\\nonumber
\geq &\ \frac{h_{11ii}}{\kappa_1}+2\sum_{p>m}\frac{h_{1pi}^2}{\kappa_1(\kappa_1-\kappa_p)}-\frac{ h_{11i}^2}{\kappa_1^2}-N\frac{u_{ii}}{u}+\alpha\Phi_{ii},
\end{align}
in the viscosity sense.

By (\ref{comm}), we have
\begin{align*}
h_{11ii}=h_{ii11}+h_{11}^2h_{ii}-h_{ii}^2h_{11}-h_{11}+h_{ii}.
\end{align*}

Plugging into (\ref{General-1}), together with Lemma \ref{Phi}, we have
\begin{align*}
0\geq&\ \frac{h_{ii11}}{\kappa_1}+2\sum_{p>m}\frac{h_{1pi}^2}{\kappa_1(\kappa_1-\kappa_p)}-\frac{ h_{11i}^2}{\kappa_1^2}-N\frac{u_{ii}}{u}\\
&\ +\alpha \left(\phi^\prime-h_{ii}u\right)+\kappa_1\kappa_i-\kappa_i^2-1+\frac{\kappa_i}{\kappa_1}.
\end{align*}

Contracting with $F^{ii}=\sigma_k^{ii}$, together with Lemma \ref{Sigma_k-Lemma-0}, we have
\begin{align}\label{General-2}
0\geq &\ \sum_i\frac{F^{ii}h_{ii11}}{\kappa_1}+2\sum_i\sum_{p>m}\frac{F^{ii}h_{1pi}^2}{\kappa_1(\kappa_1-\kappa_p)}-\sum_i\frac{ F^{ii}h_{11i}^2}{\kappa_1^2}-N\sum_i\frac{F^{ii}u_{ii}}{u}\\\nonumber
&\ +\alpha \phi^\prime\sum_i F^{ii}-\alpha kFu+kF\kappa_1-\sum_iF^{ii}\kappa_i^2-\sum_i F^{ii}+\frac{kF}{\kappa_1}\\\nonumber
\geq &\ \sum_i\frac{F^{ii}h_{ii11}}{\kappa_1}+2\sum_i\sum_{p>m}\frac{F^{ii}h_{1pi}^2}{\kappa_1(\kappa_1-\kappa_p)}-\sum_i\frac{ F^{ii}h_{11i}^2}{\kappa_1^2}-N\sum_i\frac{F^{ii}u_{ii}}{u}\\\nonumber
&\ +\left(\alpha \phi^\prime-1\right) \sum_iF^{ii}-\sum_iF^{ii}\kappa_i^2-C\alpha,
\end{align}
where $C$ is a universal constant depending only on $n,k,\|M\|_{C^1}, \inf f$ and $\|f\|_{C^2}$. From now on, we will use $C$ to denote a universal constant depending only on $n,k,\|M\|_{C^1}, \inf f$ and $\|f\|_{C^2}$, it may change from line to line.

By Lemma \ref{support function}, we have
\begin{align*}
u_{ii}=\sum_kh_{iik}\Phi_k+\phi^\prime h_{ii}-h_{ii}^2u.
\end{align*}

Together with Lemma \ref{Sigma_k-Lemma-0}, we have
\begin{align*}
-N\sum_i\frac{F^{ii}u_{ii}}{u}=&\ -N\sum_k\frac{F_k\Phi_k}{u}-N\frac{\phi^\prime kF}{u}+ N\sum_iF^{ii}\kappa_i^2\\
\geq &\ -N\sum_k\frac{F_k\Phi_k}{u} + N\sum_iF^{ii}\kappa_i^2-CN.
\end{align*}

Plugging into (\ref{General-2}), we have
\begin{align}\label{General-3}
0\geq&\ \sum_i\frac{F^{ii}h_{ii11}}{\kappa_1}+2\sum_i\sum_{p>m}\frac{F^{ii}h_{1pi}^2}{\kappa_1(\kappa_1-\kappa_p)}-\sum_i\frac{ F^{ii}h_{11i}^2}{\kappa_1^2}-N\sum_k\frac{F_k\Phi_k}{u}\\\nonumber
&\ +\left(\alpha \phi^\prime-1\right) \sum_iF^{ii}+(N-1)\sum_iF^{ii}\kappa_i^2-C\alpha-CN.
\end{align}

Differentiating equation (\ref{Equation-Sigma_k}) twice, we have
\begin{align*}
&\ \sum_iF^{ii}h_{ii11}+\sum_{p,q,r,s}F^{pq,rs}h_{pq1}h_{rs1}=f_{11}\\
=&\ \left(\sum_kh_{k1}(d_\nu f)(e_k)+ (d_Xf)(X_1)\right)_1 \\
=&\ \sum_kh_{k11}(d_\nu f)(e_k)+h_{11}^2(d_{\nu\nu} f)(e_1,e_1)-h_{11}^2(d_\nu f)(\nu)\\
&\ +(d_{XX}f)(X_1,X_1)-h_{11}(d_Xf)(\nu)\\
\geq &\ \sum_k h_{k11}(d_\nu f)(e_k)-C\kappa_1^2-C\kappa_1-C.
\end{align*}

Plugging into (\ref{General-3}), we have
\begin{align}\label{General-4}
0\geq&\ -\sum_{p,q,r,s}\frac{F^{pq,rs}h_{pq1}h_{rs1}}{\kappa_1}+2\sum_i\sum_{p>m}\frac{F^{ii}h_{1pi}^2}{\kappa_1(\kappa_1-\kappa_p)}-\sum_i\frac{ F^{ii}h_{11i}^2}{\kappa_1^2}\\\nonumber
&\ -N\sum_k\frac{F_k\Phi_k}{u}+\left(\alpha \phi^\prime-1\right) \sum_iF^{ii}+(N-1)\sum_iF^{ii}\kappa_i^2\\\nonumber
&\ +\sum_k\frac{h_{11k}}{\kappa_1}(d_\nu f)(e_k)-C\kappa_1-C\alpha-CN.
\end{align}

By Lemma \ref{support function} and the critical equation (\ref{General-critical}), we have
\begin{align*}
&\ -N\sum_k\frac{F_k\Phi_k}{u}+\sum_k\frac{h_{11k}}{\kappa_1}(d_\nu f)(e_k)\\
=&\ -N\sum_k\bigg(h_{kk}(d_\nu f)(e_k)+ (d_Xf)(X_k)\bigg)\frac{\Phi_k}{u}+\sum_k\left( N\frac{u_k}{u}-\alpha\Phi_k\right) (d_\nu f)(e_k)\\
\geq &\ -N \sum_kh_{kk}(d_\nu f)(e_k)\frac{\Phi_k}{u}+N\sum_k\frac{h_{kk}\Phi_k}{u}(d_\nu f)(e_k)-CN-C\alpha\\
= &\ -CN-C\alpha.
\end{align*}

Plugging into (\ref{General-4}), we have
\begin{align*}
0\geq&\ -\sum_{p,q,r,s}\frac{F^{pq,rs}h_{pq1}h_{rs1}}{\kappa_1}+2\sum_i\sum_{p>m}\frac{F^{ii}h_{1pi}^2}{\kappa_1(\kappa_1-\kappa_p)}-\sum_i\frac{ F^{ii}h_{11i}^2}{\kappa_1^2}\\\nonumber
&\ +\left(\alpha \phi^\prime-1\right) \sum_iF^{ii}+(N-1)\sum_iF^{ii}\kappa_i^2-C\kappa_1-C\alpha-CN.
\end{align*}

Now 
\begin{align*}
-\sum_{p,q,r,s} F^{pq,rs}h_{pq1}h_{rs1}=-\sum_{p\neq q}F^{pp,qq}h_{pp1}h_{qq1}+\sum_{p\neq q}F^{pp,qq}h_{pq1}^2.
\end{align*}

Thus
\begin{align}\label{General-5}
0\geq &\ -\sum_{p\neq q}\frac{F^{pp,qq}h_{pp1}h_{qq1}}{\kappa_1}+\sum_{p\neq q} \frac{F^{pp,qq}h_{pq1}^2}{\kappa_1}+2\sum_i\sum_{p>m}\frac{F^{ii}h_{1pi}^2}{\kappa_1(\kappa_1-\kappa_p)}\\\nonumber
&\ -\sum_i\frac{ F^{ii}h_{11i}^2}{\kappa_1^2}+\left(\alpha \phi^\prime-1\right) \sum_iF^{ii}+(N-1)\sum_iF^{ii}\kappa_i^2\\\nonumber
&\ -C\kappa_1-C\alpha-CN.
\end{align}

Let us concentrate on the third order terms. By Lemma \ref{Sigma_k-Lemma-0}, we have
\begin{align}\label{General-6}
\sum_{p\neq q} \frac{F^{pp,qq}h_{pq1}^2}{\kappa_1}\geq 2\sum_{i>m}\frac{F^{11,ii}h_{11i}^2}{\kappa_1}= 2\sum_{i>m}\frac{(F^{ii}-F^{11})h_{11i}^2}{\kappa_1(\kappa_1-\kappa_i)}.
\end{align}

On the other hand
\begin{align}\label{General-7}
2\sum_i\sum_{p>m}\frac{F^{ii}h_{1pi}^2}{\kappa_1(\kappa_1-\kappa_p)}\geq&\ 2\sum_{p>m}\frac{F^{pp}h_{1pp}^2}{\kappa_1(\kappa_1-\kappa_p)}+2\sum_{p>m}\frac{F^{11}h_{1p1}^2}{\kappa_1(\kappa_1-\kappa_p)}\\\nonumber
=&\ 2\sum_{i>m}\frac{F^{ii}h_{ii1}^2}{\kappa_1(\kappa_1-\kappa_i)}+2\sum_{i>m}\frac{F^{11}h_{11i}^2}{\kappa_1(\kappa_1-\kappa_i)}.
\end{align}

By (\ref{BCD}), we have
\begin{align}\label{General-8}
h_{11i}=h_{1i1}=\delta_{li}\cdot {\kappa_1}_1=0,\quad 1<i\leq m.
\end{align}

Plugging (\ref{General-6}), (\ref{General-7}) and (\ref{General-8}) into (\ref{General-5}), we have
\begin{align*}
0\geq &\ -\sum_{p\neq q}\frac{F^{pp,qq}h_{pp1}h_{qq1}}{\kappa_1}+ 2\sum_{i>m}\frac{F^{ii}h_{11i}^2}{\kappa_1(\kappa_1-\kappa_i)}+2\sum_{i>m}\frac{F^{ii}h_{ii1}^2}{\kappa_1(\kappa_1-\kappa_i)}-\sum_i\frac{ F^{ii}h_{11i}^2}{\kappa_1^2}\\
&\ +\left(\alpha \phi^\prime-1\right) \sum_iF^{ii}+(N-1)\sum_iF^{ii}\kappa_i^2-C\kappa_1-C\alpha-CN\\
=&\ -\sum_{p\neq q}\frac{F^{pp,qq}h_{pp1}h_{qq1}}{\kappa_1}+\sum_{i>m}\frac{F^{ii}(\kappa_1+\kappa_i)h_{11i}^2}{\kappa_1^2(\kappa_1-\kappa_i)}+2\sum_{i>m}\frac{F^{ii}h_{ii1}^2}{\kappa_1(\kappa_1-\kappa_i)}-\frac{ F^{11}h_{111}^2}{\kappa_1^2}\\
&\ +\left(\alpha \phi^\prime-1\right) \sum_iF^{ii}+(N-1)\sum_iF^{ii}\kappa_i^2-C\kappa_1-C\alpha-CN.
\end{align*}

Recall that $\kappa_i\geq -K$. Without loss of generality, we may assume $\kappa_1\geq K$. Thus $\kappa_1+\kappa_i\geq 0$. Therefore,
\begin{align}\label{General-9}
0\geq &\ -\sum_{p\neq q}\frac{F^{pp,qq}h_{pp1}h_{qq1}}{\kappa_1}+2\sum_{i>m}\frac{F^{ii}h_{ii1}^2}{\kappa_1(\kappa_1-\kappa_i)}-\frac{ F^{11}h_{111}^2}{\kappa_1^2}\\\nonumber
&\ +\left(\alpha \phi^\prime-1\right) \sum_iF^{ii}+(N-1)\sum_iF^{ii}\kappa_i^2-C\kappa_1-C\alpha-CN.
\end{align}

\medskip

We will now use Lemma \ref{Concavity-Lemma} to take care of the term $-\frac{ F^{11}h_{111}^2}{\kappa_1^2}$. 

{\bf Case 1:} There exists $m\leq l<k$ such that $\kappa_l\geq \delta \kappa_1$ and $\kappa_{l+1}\leq \delta^\prime \kappa_1$, where we have chosen $\delta_0=\frac{1}{2}$ and $\epsilon$ is a small constant to be determined later.

In this case, by Lemma \ref{Concavity-Lemma}, we have
\begin{align}\label{General-10}
-\sum_{p\neq q} F^{pp,qq}h_{pp1}h_{qq1}+\frac{\left( \sum_i F^{ii}h_{ii1}\right)^2}{\sigma_k}\geq (1-\epsilon)\frac{\sigma_k h_{111}^2}{\kappa_1^2}-\frac{1}{2}\sum_{i>l}\frac{F^{ii}h_{ii1}^2}{\kappa_1 },
\end{align}

Without loss of generality, we may assume $\kappa_1\geq K$. Since $\kappa_i\geq -K$, we have
\begin{align}\label{General-11}
\frac{2}{\kappa_1(\kappa_1-\kappa_i)}-\frac{1}{2\kappa_1^2}=\frac{3\kappa_1+\kappa_i}{2\kappa_1^2(\kappa_1-\kappa_i)}\geq 0.
\end{align}

Plugging (\ref{General-10}) and (\ref{General-11}) into (\ref{General-9}), together with Lemma \ref{Sigma_k-Lemma-0}, we have
\begin{align}\label{General-12}
0\geq &\ -\frac{F_1^2}{\kappa_1\sigma_k}+ (1-\epsilon)\frac{\sigma_k h_{111}^2}{\kappa_1^3}-\frac{ F^{11}h_{111}^2}{\kappa_1^2}+\left(\alpha \phi^\prime-1\right) \sum_iF^{ii}\\\nonumber
&\ +(N-1)\sum_iF^{ii}\kappa_i^2-C\kappa_1-C\alpha-CN\\\nonumber
\geq &\ (1-\epsilon)\sigma_k(\kappa|1) \frac{h_{111}^2}{\kappa_1^3}-\epsilon\frac{ F^{11}h_{111}^2}{\kappa_1^2}+\left(\alpha \phi^\prime-1\right) \sum_iF^{ii}\\\nonumber
&\ +(N-1)\sum_iF^{ii}\kappa_i^2-C\kappa_1-C\alpha-CN.
\end{align}

By the critical equation (\ref{General-critical}) and the fact that $\kappa_i\geq -K$, we have
\begin{align}\label{General-13}
(1-\epsilon)\sigma_k(\kappa|1) \frac{h_{111}^2}{\kappa_1^3}\geq&\  -C\kappa_2\cdots\kappa_k\cdot K\cdot \frac{1}{\kappa_1}\cdot\left(N\frac{u_1}{u}-\alpha\Phi_1\right)^2\\\nonumber
\geq&\ -CN^2 \kappa_1\cdots\kappa_k-C\alpha^2 \frac{\kappa_2\cdots\kappa_k}{\kappa_1}\\\nonumber
\geq &\ -CN^2 \kappa_1\cdots\kappa_k,
\end{align}
by assuming $\kappa_1$ sufficiently large.

Similarly,
\begin{align}\label{General-14}
-\epsilon\frac{F^{11}h_{111}^2}{\kappa_1^2}=&\ -\epsilon F^{11}  \left( N\frac{u_1}{u}-\alpha\Phi_1\right)^2\\\nonumber
\geq&\ -C\epsilon N^2 F^{11}\kappa_1^2-C\epsilon \alpha^2 F^{11}\\\nonumber
\geq&\ -C\epsilon N^2 F^{11}\kappa_1^2,
\end{align}
by assuming $\kappa_1$ sufficiently large.

By Lemma \ref{Sigma_k-Lemma-0} and Lemma \ref{Sigma_k-lemma} and the fact that $\phi^\prime(r)=\cosh(r)\geq 1$, we have
\begin{align}\label{General-15}
\left(\alpha \phi^\prime-1\right) \sum_iF^{ii}=\left(\alpha \phi^\prime-1\right) (n-k)\sigma_{k-1}\geq C\alpha \kappa_1\cdots\kappa_{k-1},
\end{align}
by choosing $\alpha$ sufficiently large.

Combining (\ref{General-13}) and (\ref{General-15}), we have
\begin{align}\label{General-16}
(1-\epsilon)\sigma_k(\kappa|1) \frac{h_{111}^2}{\kappa_1^3}+\left(\alpha \phi^\prime-1\right) \sum_iF^{ii}\geq \left(C\alpha-CN^2\kappa_k\right)\kappa_1\cdots\kappa_{k-1}.
\end{align}

From now on we will fix $\alpha=N^3$.

\medskip

If $\kappa_k\geq \frac{C\alpha}{N^2}=CN$, then we have
\begin{align*}
\sigma_k\geq&\  \kappa_1\cdots\kappa_k-C\kappa_1\cdots\kappa_{k-1}\cdot K\\
\geq &\ \kappa_1\cdots\kappa_{k-1}\left(CN-C\right)\\
\geq &\ C\kappa_1,
\end{align*}
by choosing $N$ sufficiently large.

It follows that $\kappa_1\leq C$.

\medskip

If $\kappa_k\leq \frac{C\alpha}{N^2}=CN$, choosing $\epsilon=\frac{1}{N^2}$, then by (\ref{General-12}), (\ref{General-14}), (\ref{General-16}) as well as Lemma \ref{Sigma_k-lemma}, we have
\begin{align*}
0\geq &\ -C\epsilon N^2 F^{11}\kappa_1^2+(N-1)\sum_iF^{ii}\kappa_i^2-C\kappa_1-C\alpha-CN\\
\geq &\ (N-C)\sum_iF^{ii}\kappa_i^2-C\kappa_1-C\alpha-CN\\
\geq &\ CN\sigma_1\sigma_k-C\kappa_1-C\alpha-CN\\
\geq&\  (CN-C)\kappa_1-C\alpha-CN\\
\geq &\ C\kappa_1-C\alpha-CN,
\end{align*}
by choosing $N$ sufficiently large.

It follows that $\kappa_1\leq C$.

\medskip

{\bf Case 2:} For each $2\leq l\leq k$, there exists $\delta_l>0$ such that $\kappa_l\geq \delta_l\kappa_1$.

In this case, we have
\begin{align*}
\sigma_k\geq&\ \kappa_1\cdots\kappa_k-C\kappa_1\cdots\kappa_{k-1}\cdot K\\
\geq&\ \kappa_1\cdots\kappa_{k-1}\left(\delta_k\kappa_1-C\right)\\
\geq&\  C\delta_2\cdots\delta_k\kappa_1^k,
\end{align*}
by assuming $\kappa_1$ sufficiently large.

It follows that $\kappa_1\leq C$.

The theorem is now proved.

\end{proof}

\section{Prescribed curvature measure type problem}

In this section, we will prove Theorem \ref{Theorem-Curvature-Measure}. It is a consequence of the following theorem.
\begin{theo}
Let $M$ be a strictly star-shaped hypersurface satisfying curvature equation (\ref{Equation-Curvature-Measure}) in $\mathbb{H}^{n+1}$ with $\kappa\in \Gamma_k$. Let $p\in (-\infty,0)\cup (0,1] $ and let $\varphi\in C^2(M)$ be a positive function. Then we have
\begin{align*}
\max_{X\in M;1\leq i\leq n}|\kappa_i(X)|\leq C\left(1+\max_{X\in \partial M;1\leq i\leq n}|\kappa_i(X)|\right),
\end{align*}
where $C$ is a constant depending only on $n,k,p,\|M\|_{C^1}, \inf \varphi$ and $\|\varphi\|_{C^2}$.
\end{theo}

\begin{proof}
Since $k=1,2$ was already solved, we only need to consider the case $k\geq 3$. Since $M$ is strictly star-shaped, without loss of generality, we may assume $u>2a>0$. 

Consider the test function
\begin{align*}
Q=\ln \kappa_1-\ln (u-a) +\alpha\Phi,
\end{align*}
where $\kappa_1$ is the largest principle curvature and $\alpha$ is a large constant to be determined later.

In the following, we will do a standard but slightly lengthy computation at the maximum point. First time readers can start from (\ref{Curvature-11}).

Assume $Q$ achieves maximum at an interior point $X_0$. At $X_0$, we can choose an orthonormal frame such that $(h_{ij})$ is diagonalized. Without loss of generality, we may assume $\kappa_1$ has multiplicity $m$, i.e.
\begin{align*}
\kappa_1=\cdots=\kappa_m>\kappa_{m+1}\geq \cdots\geq \kappa_n.
\end{align*}

By Lemma 5 in \cite{BCD}, at $X_0$, we have
\begin{align}\label{BCD-1}
\delta_{kl}\cdot{\kappa_1}_i=h_{kli},\quad 1\leq k,l\leq m,
\end{align}
\begin{align*}
{\kappa_1}_{ii}\geq h_{11ii}+2\sum_{p>m}\frac{h_{1pi}^2}{\kappa_1 -\kappa_p},
\end{align*}
in the viscosity sense.

At $X_0$, we have
\begin{align}\label{Curvature-critical}
0=\frac{{\kappa_1}_i}{\kappa_1}-\frac{u_i}{u-a}+\alpha\Phi_i=\frac{h_{11i}}{\kappa_1}-\frac{u_i}{u-a}+\alpha\Phi_i,
\end{align}
\begin{align}\label{Curvature-1}
0\geq &\ \frac{{\kappa_1}_{ii}}{\kappa_1}-\frac{ (\kappa_1)_i^2}{\kappa_1^2}-\frac{u_{ii}}{u-a}+\frac{u_i^2}{(u-a)^2}+\alpha\Phi_{ii}\\\nonumber
\geq &\ \frac{h_{11ii}}{\kappa_1}+2\sum_{p>m}\frac{h_{1pi}^2}{\kappa_1(\kappa_1-\kappa_p)}-\frac{ h_{11i}^2}{\kappa_1^2}-\frac{u_{ii}}{u-a}+\frac{u_i^2}{(u-a)^2}+\alpha\Phi_{ii},
\end{align}
in the viscosity sense.

By (\ref{comm}), we have
\begin{align*}
h_{11ii}=h_{ii11}+h_{11}^2h_{ii}-h_{ii}^2h_{11}-h_{11}+h_{ii}.
\end{align*}

Plugging into (\ref{Curvature-1}), together with Lemma \ref{Phi}, we have
\begin{align*}
0\geq&\ \frac{h_{ii11}}{\kappa_1}+2\sum_{p>m}\frac{h_{1pi}^2}{\kappa_1(\kappa_1-\kappa_p)}-\frac{ h_{11i}^2}{\kappa_1^2}-\frac{u_{ii}}{u-a}+\frac{u_i^2}{(u-a)^2}\\
&\ +\alpha \left(\phi^\prime-h_{ii}u\right)+\kappa_1\kappa_i-\kappa_i^2-1+\frac{\kappa_i}{\kappa_1}.
\end{align*}

Contracting with $F^{ii}=\sigma_k^{ii}$, together with Lemma \ref{Sigma_k-Lemma-0}, we have
\begin{align}\label{Curvature-2}
0\geq &\ \sum_i\frac{F^{ii}h_{ii11}}{\kappa_1}+2\sum_i\sum_{p>m}\frac{F^{ii}h_{1pi}^2}{\kappa_1(\kappa_1-\kappa_p)}-\sum_i\frac{ F^{ii}h_{11i}^2}{\kappa_1^2}-\sum_i\frac{F^{ii}u_{ii}}{u-a}\\\nonumber
&\ +\sum_i\frac{F^{ii}u_i^2}{(u-a)^2}+\alpha \phi^\prime\sum_i F^{ii}-\alpha kFu+kF\kappa_1-\sum_iF^{ii}\kappa_i^2-\sum_i F^{ii}+\frac{kF}{\kappa_1}\\\nonumber
\geq &\ \sum_i\frac{F^{ii}h_{ii11}}{\kappa_1}+2\sum_i\sum_{p>m}\frac{F^{ii}h_{1pi}^2}{\kappa_1(\kappa_1-\kappa_p)}-\sum_i\frac{ F^{ii}h_{11i}^2}{\kappa_1^2}-\sum_i\frac{F^{ii}u_{ii}}{u-a}\\\nonumber
&\ +\sum_i\frac{F^{ii}u_i^2}{(u-a)^2}+\left(\alpha \phi^\prime-1\right) \sum_iF^{ii}+kF\kappa_1-\sum_iF^{ii}\kappa_i^2-C\alpha,
\end{align}
where $C$ is a universal constant depending only on $n,k,p,\|M\|_{C^1}, \inf \varphi$ and $\|\varphi\|_{C^2}$. From now on, we will use $C$ to denote a universal constant depending only on $n,k,p,\|M\|_{C^1}, \inf \varphi$ and $\|\varphi\|_{C^2}$, it may change from line to line.

By Lemma \ref{support function}, we have
\begin{align*}
u_{ii}=\sum_kh_{iik}\Phi_k+\phi^\prime h_{ii}-h_{ii}^2u.
\end{align*}

Together with Lemma \ref{Sigma_k-Lemma-0}, we have
\begin{align*}
-\sum_i\frac{F^{ii}u_{ii}}{u-a}=&\ -\sum_k\frac{F_k\Phi_k}{u-a}-\frac{\phi^\prime kF}{u-a}+ \frac{u}{u-a}\sum_iF^{ii}\kappa_i^2\\
\geq &\ -\sum_k\frac{F_k\Phi_k}{u-a} + \frac{u}{u-a}\sum_iF^{ii}\kappa_i^2-C.
\end{align*}

Plugging into (\ref{Curvature-2}), we have
\begin{align}\label{Curvature-3}
0\geq&\ \sum_i\frac{F^{ii}h_{ii11}}{\kappa_1}+2\sum_i\sum_{p>m}\frac{F^{ii}h_{1pi}^2}{\kappa_1(\kappa_1-\kappa_p)}-\sum_i\frac{ F^{ii}h_{11i}^2}{\kappa_1^2}-\sum_k\frac{F_k\Phi_k}{u-a}\\\nonumber
&\ +\sum_i\frac{F^{ii}u_i^2}{(u-a)^2}+\left(\alpha \phi^\prime-1\right) \sum_iF^{ii}+kF\kappa_1+ \frac{a}{u-a}\sum_iF^{ii}\kappa_i^2-C\alpha-C.
\end{align}

Differentiating equation (\ref{Equation-Curvature-Measure}) twice, together with Lemma \ref{support function}, we have
\begin{align*}
&\ \sum_i F^{ii}h_{ii11}+\sum_{p,q,r,s} F^{pq,rs}h_{pq1}h_{rs1}= (u^p\varphi)_{11}  \\
=&\ \varphi\bigg( pu^{p-1} u_{11}+ p(p-1)u^{p-2} u_1^2\bigg)+2\varphi_1\cdot pu^{p-1} u_1+u^p\cdot\varphi_{11}\\
=& \ \varphi \cdot pu^{p-1}\left(\sum_k h_{11k}\Phi_k+\phi^\prime h_{11}-h_{11}^2u\right)+\varphi\cdot p(p-1)u^{p-2}u_1^2-C\kappa_1-C\\
\geq &\ \frac{pF}{u}\sum_k h_{11k}\Phi_k-pF \kappa_1^2 + p(p-1)F \frac{ u_1^2}{u^2}-C\kappa_1-C.
\end{align*}

Plugging into (\ref{Curvature-3}), we have
\begin{align}\label{Curvature-4}
0\geq&\ -\sum_{p,q,r,s}\frac{F^{pq,rs}h_{pq1}h_{rs1}}{\kappa_1}+2\sum_i\sum_{p>m}\frac{F^{ii}h_{1pi}^2}{\kappa_1(\kappa_1-\kappa_p)}-\sum_i\frac{ F^{ii}h_{11i}^2}{\kappa_1^2}\\\nonumber
&\ -\sum_k\frac{F_k\Phi_k}{u-a}+\sum_i\frac{F^{ii}u_i^2}{(u-a)^2}+\left(\alpha \phi^\prime-1\right) \sum_iF^{ii}+ \frac{a}{u-a}\sum_iF^{ii}\kappa_i^2\\\nonumber
&\ +\frac{pF}{u}\sum_k\frac{h_{11k}}{\kappa_1}\Phi_k+(k-p)F \kappa_1 +p(p-1)\frac{F}{\kappa_1} \frac{ u_1^2}{u^2}-C\alpha-C.
\end{align}

By the critical equation (\ref{Curvature-critical}), we have
\begin{align*}
&\ -\sum_k\frac{F_k\Phi_k}{u-a}+\frac{pF}{u}\sum_k\frac{h_{11k}}{\kappa_1}\Phi_k\\
=&\ -\sum_k\bigg(\varphi\cdot pu^{p-1}u_k+u^p\cdot \varphi_k\bigg)\frac{\Phi_k}{u-a}+\frac{pF}{u}\sum_k\left( \frac{u_k}{u-a}-\alpha\Phi_k\right)\Phi_k\\
\geq &\ -\frac{pF}{u}\sum_k\frac{u_k\Phi_k}{u-a}+\frac{pF}{u}\sum_k\frac{u_k\Phi_k}{u-a}-C\alpha-C\\
= &\ -C\alpha-C.
\end{align*}

Plugging into (\ref{Curvature-4}), we have
\begin{align*}
0\geq&\ -\sum_{p,q,r,s}\frac{F^{pq,rs}h_{pq1}h_{rs1}}{\kappa_1}+2\sum_i\sum_{p>m}\frac{F^{ii}h_{1pi}^2}{\kappa_1(\kappa_1-\kappa_p)}-\sum_i\frac{ F^{ii}h_{11i}^2}{\kappa_1^2}\\\nonumber
&\ +\sum_i\frac{F^{ii}u_i^2}{(u-a)^2}+\left(\alpha \phi^\prime-1\right) \sum_iF^{ii}+ \frac{a}{u-a}\sum_iF^{ii}\kappa_i^2\\\nonumber
&\ +(k-p)F \kappa_1 +p(p-1)\frac{F}{\kappa_1} \frac{ u_1^2}{u^2}-C\alpha-C.
\end{align*}

Now 
\begin{align*}
-\sum_{p,q,r,s} F^{pq,rs}h_{pq1}h_{rs1}=-\sum_{p\neq q}F^{pp,qq}h_{pp1}h_{qq1}+\sum_{p\neq q}F^{pp,qq}h_{pq1}^2.
\end{align*}

Thus
\begin{align}\label{Curvature-5}
0\geq&\ -\sum_{p\neq q} \frac{F^{pp,qq}h_{pp1}h_{qq1}}{\kappa_1} +\sum_{p\neq q} \frac{F^{pp,qq}h_{pq1}^2}{\kappa_1} +2\sum_i\sum_{p>m}\frac{F^{ii}h_{1pi}^2}{\kappa_1(\kappa_1-\kappa_p)}\\\nonumber
&\ -\sum_i\frac{ F^{ii}h_{11i}^2}{\kappa_1^2}+\sum_i\frac{F^{ii}u_i^2}{(u-a)^2}+\left(\alpha \phi^\prime-1\right) \sum_iF^{ii}+ \frac{a}{u-a}\sum_iF^{ii}\kappa_i^2\\\nonumber
&\ +(k-p)F \kappa_1 +p(p-1)\frac{F}{\kappa_1} \frac{ u_1^2}{u^2}-C\alpha-C.
\end{align}

Let us concentrate on the third order terms. By Lemma \ref{Sigma_k-Lemma-0}, we have
\begin{align}\label{Curvature-6}
\sum_{p\neq q} \frac{F^{pp,qq}h_{pq1}^2}{\kappa_1}\geq 2\sum_{i>m}\frac{F^{11,ii}h_{11i}^2}{\kappa_1}= 2\sum_{i>m}\frac{(F^{ii}-F^{11})h_{11i}^2}{\kappa_1(\kappa_1-\kappa_i)}.
\end{align}

On the other hand
\begin{align}\label{Curvature-7}
2\sum_i\sum_{p>m}\frac{F^{ii}h_{1pi}^2}{\kappa_1(\kappa_1-\kappa_p)}\geq&\ 2\sum_{p>m}\frac{F^{pp}h_{1pp}^2}{\kappa_1(\kappa_1-\kappa_p)}+2\sum_{p>m}\frac{F^{11}h_{1p1}^2}{\kappa_1(\kappa_1-\kappa_p)}\\\nonumber
=&\ 2\sum_{i>m}\frac{F^{ii}h_{ii1}^2}{\kappa_1(\kappa_1-\kappa_i)}+2\sum_{i>m}\frac{F^{11}h_{11i}^2}{\kappa_1(\kappa_1-\kappa_i)}.
\end{align}

By (\ref{BCD-1}), we have
\begin{align}\label{Curvature-8}
h_{11i}=h_{1i1}=\delta_{li}\cdot {\kappa_1}_1=0,\quad 1<i\leq m.
\end{align}

Plugging (\ref{Curvature-6}), (\ref{Curvature-7}) and (\ref{Curvature-8}) into (\ref{Curvature-5}), we have
\begin{align}\label{Curvature-9}
0\geq &\ -\sum_{p\neq q}\frac{F^{pp,qq}h_{pp1}h_{qq1}}{\kappa_1}+ 2\sum_{i>m}\frac{F^{ii}h_{11i}^2}{\kappa_1(\kappa_1-\kappa_i)}+2\sum_{i>m}\frac{F^{ii}h_{ii1}^2}{\kappa_1(\kappa_1-\kappa_i)}\\\nonumber
&\ -\sum_i\frac{ F^{ii}h_{11i}^2}{\kappa_1^2}+\sum_i\frac{F^{ii}u_i^2}{(u-a)^2}+\left(\alpha \phi^\prime-1\right) \sum_iF^{ii}+ \frac{a}{u-a}\sum_iF^{ii}\kappa_i^2\\\nonumber
&\ +(k-p)F \kappa_1 +p(p-1)\frac{F}{\kappa_1}  \frac{ u_1^2}{u^2}-C\alpha-C\\\nonumber
=&\ -\sum_{p\neq q}\frac{F^{pp,qq}h_{pp1}h_{qq1}}{\kappa_1}+\sum_{i>m}\frac{F^{ii}(\kappa_1+\kappa_i)h_{11i}^2}{\kappa_1^2(\kappa_1-\kappa_i)}+2\sum_{i>m}\frac{F^{ii}h_{ii1}^2}{\kappa_1(\kappa_1-\kappa_i)}\\\nonumber
&\ -\frac{ F^{11}h_{111}^2}{\kappa_1^2}+\sum_i\frac{F^{ii}u_i^2}{(u-a)^2}+\left(\alpha \phi^\prime-1\right) \sum_iF^{ii}+ \frac{a}{u-a}\sum_iF^{ii}\kappa_i^2\\\nonumber
&\ +(k-p)F \kappa_1 +p(p-1)\frac{F}{\kappa_1}  \frac{ u_1^2}{u^2}-C\alpha-C.
\end{align}

By the critical equation (\ref{Curvature-critical}), we have
\begin{align}\label{Curvature-10}
&\ \frac{a}{n(u-a)}F^{11}\kappa_1^2-\frac{ F^{11}h_{111}^2}{\kappa_1^2}+\frac{F^{11}u_1^2}{(u-a)^2}\\\nonumber
=&\ \frac{a}{n(u-a)}F^{11}\kappa_1^2-F^{11}\left(\frac{u_1}{u-a}-\alpha \Phi_1\right)^2+\frac{F^{11}u_1^2}{(u-a)^2}\\\nonumber
=&\ \frac{a}{n(u-a)}F^{11}\kappa_1^2+2\alpha F^{11}\frac{u_1}{u-a} \Phi_1-\alpha^2F^{11} \Phi_1^2\\\nonumber
\geq&\ \frac{a}{n(u-a)}F^{11}\kappa_1^2-C\alpha F^{11}\kappa_1-C\alpha^2 F^{11}\geq 0,
\end{align}
by assuming $\kappa_1$ sufficiently large.

Plugging (\ref{Curvature-10}) into (\ref{Curvature-9}), we have
\begin{align}\label{Curvature-11}
0\geq &\ -\sum_{p\neq q}\frac{F^{pp,qq}h_{pp1}h_{qq1}}{\kappa_1}+\sum_{i>m}\frac{F^{ii}(\kappa_1+\kappa_i)h_{11i}^2}{\kappa_1^2(\kappa_1-\kappa_i)}+2\sum_{i>m}\frac{F^{ii}h_{ii1}^2}{\kappa_1(\kappa_1-\kappa_i)}\\\nonumber
&\ +\sum_{i>1}\frac{F^{ii}u_i^2}{(u-a)^2}+\left(\alpha \phi^\prime-1\right) \sum_iF^{ii}+ \frac{n-1}{n}\frac{a}{u-a}\sum_iF^{ii}\kappa_i^2\\\nonumber
&\ +(k-p)F \kappa_1 +p(p-1)\frac{F}{\kappa_1}  \frac{ u_1^2}{ u^2}-C\alpha-C.
\end{align}

\medskip

{\bf Case 1:} There exists $\kappa_i$ such that $\kappa_1+\kappa_i\leq 0$. 

In this case, for each $i$ with $\kappa_1+\kappa_i\leq 0$, by the critical equation (\ref{Curvature-critical}), we have
\begin{align}\label{Curvature-12}
&\ \frac{a}{n(u-a)}F^{ii}\kappa_i^2+\frac{F^{ii}(\kappa_1+\kappa_i)h_{11i}^2}{\kappa_1^2(\kappa_1-\kappa_i)}+\frac{F^{ii}u_i^2}{(u-a)^2}\\\nonumber
\geq&\ \frac{a}{n(u-a)}F^{ii}\kappa_i^2-\frac{F^{ii}h_{11i}^2}{\kappa_1^2}+\frac{F^{ii}u_i^2}{(u-a)^2}\\\nonumber
=&\ \frac{a}{n(u-a)}F^{ii}\kappa_i^2-F^{ii}\left(\frac{u_i}{u-a}-\alpha \Phi_i\right)^2+\frac{F^{ii}u_i^2}{(u-a)^2}\\\nonumber
=&\ \frac{a}{n(u-a)}F^{ii}\kappa_i^2+2\alpha F^{ii}\frac{u_i}{u-a} \Phi_i-\alpha^2F^{ii} \Phi_i^2\\\nonumber
\geq&\ \frac{a}{n(u-a)}F^{ii}\kappa_i^2 -C\alpha F^{ii}|\kappa_i|-C\alpha^2F^{ii}\geq 0,
\end{align}
by assuming $\kappa_1$ sufficiently large as well as the fact $|\kappa_i|\geq \kappa_1$.

Plugging (\ref{Curvature-12}) into (\ref{Curvature-11}), we have
\begin{align}\label{Curvature-13}
0\geq &\ -\sum_{p\neq q}\frac{F^{pp,qq}h_{pp1}h_{qq1}}{\kappa_1}+\left(\alpha \phi^\prime-1\right) \sum_iF^{ii}+ \frac{a}{n(u-a)}\sum_iF^{ii}\kappa_i^2\\\nonumber
&\ +p(p-1)\frac{F}{\kappa_1} \frac{ u_1^2}{u^2}-C\alpha-C.
\end{align}
We have used the fact that $k\geq 3$ and there are at most $n-k$ negative $\kappa_i$'s above.

By concavity of $\sigma_k^{\frac{1}{k}}$, we have
\begin{align*}
\frac{1}{k}\sigma_k^{\frac{1}{k}-1}\sum_{p\neq q} \sigma_k^{pp,qq}h_{pp1}h_{qq1}+\frac{1}{k}\left(\frac{1}{k}-1\right)\sigma_k^{\frac{1}{k}-2}\sum_{p,q}\sigma_k^{pp}h_{pp1}\sigma_k^{qq}h_{qq1}\leq 0,
\end{align*}
i.e.
\begin{align}\label{Curvature-13-1}
-\sum_{p\neq q} F^{pp,qq}h_{pp1}h_{qq1}\geq -\frac{k-1}{k}\frac{F^2_1}{F}\geq -C\kappa_1^2.
\end{align}

Plugging into (\ref{Curvature-13}) and choosing $\alpha\geq 1$, together with the fact that $\phi^\prime(r)=\cosh(r)\geq 1$, we have
\begin{align*}
0\geq \frac{a}{n(u-a)}\sum_i F^{ii}\kappa_i^2-C\kappa_1-C\alpha-C.
\end{align*}

By Lemma \ref{Sigma_k-lemma}, we know that $\kappa_k\geq C|\kappa_n|\geq C\kappa_1$. Therefore, 
\begin{align*}
F^{nn}\kappa_n^2\geq C\sigma_{k-1}\kappa_n^2\geq C\kappa_1\cdots \kappa_{k-1}\cdot \kappa_1^2\geq C\kappa_1^{k+1}.
\end{align*}

Consequently,
\begin{align*}
0\geq C\kappa_1^{k+1}-C\kappa_1-C\alpha-C.
\end{align*}

It follows that $\kappa_1 \leq C$.

\medskip

{\bf Case 2:} $\kappa_1+\kappa_i\geq 0$ for all $i$.

In this case, by (\ref{Curvature-11}), we have
\begin{align}\label{Curvature-14}
0\geq &\ -\sum_{p\neq q}\frac{F^{pp,qq}h_{pp1}h_{qq1}}{\kappa_1}+2\sum_{i>m}\frac{F^{ii}h_{ii1}^2}{\kappa_1(\kappa_1-\kappa_i)}\\\nonumber
&\ +\left(\alpha \phi^\prime-1\right) \sum_iF^{ii}+ \frac{n-1}{n}\frac{a}{u-a}\sum_iF^{ii}\kappa_i^2\\\nonumber
&\ +(k-p)F \kappa_1 +p(p-1)\frac{F}{\kappa_1}  \frac{ u_1^2}{u^2}-C\alpha-C.
\end{align}

By Lemma \ref{Sigma_k-Lemma-0} and Lemma \ref{Sigma_k-lemma} and the fact that $\phi^\prime(r)=\cosh(r)\geq 1$, we have
\begin{align}\label{Curvature-15}
\left(\alpha \phi^\prime-1\right) \sum_iF^{ii}=\left(\alpha \phi^\prime-1\right) (n-k)\sigma_{k-1}\geq C\alpha \kappa_1\cdots\kappa_{k-1},
\end{align}
by choosing $\alpha$ sufficiently large.

Suppose now
\begin{align*}
\kappa_2\cdots \kappa_{k-1}\geq \alpha^{-\frac{k-2}{n}}.
\end{align*}

Then by combining (\ref{Curvature-13-1}) and (\ref{Curvature-15}), we have
\begin{align*}
&\ -\sum_{p\neq q}\frac{F^{pp,qq}h_{pp1}h_{qq1}}{\kappa_1}+\left(\alpha \phi^\prime-1\right) \sum_iF^{ii}+p(p-1)\frac{F}{\kappa_1}  \frac{ u_1^2}{u^2}\\
\geq&\ -C\kappa_1+C\alpha^{\frac{n-k+2}{n}} \kappa_1\geq 0,
\end{align*}
by choosing $\alpha$ sufficiently large.

Plugging into (\ref{Curvature-14}), we have
\begin{align*}
0\geq & (k-p)F \kappa_1 -C\alpha-C.
\end{align*}

It follows that $\kappa_1\leq C$.

\medskip

In the following, we will assume
\begin{align}\label{Curvature-15-1}
\kappa_2\cdots \kappa_{k-1}\leq \alpha^{-\frac{k-2}{n}}.
\end{align}

Consequently,
\begin{align*}
|\sigma_k(\kappa|1)|\leq C\kappa_2 \cdots \kappa_k \cdot \kappa_k\leq C\alpha^{-\frac{k}{n}}.
\end{align*}

By Lemma \ref{Sigma_k-Lemma-0}, we have
\begin{align*}
\sigma_k=\kappa_1\sigma_k^{11}+\sigma_k(\kappa|1).
\end{align*}

It follows that
\begin{align}\label{Curvature-16}
\frac{1}{2}\sigma_k\leq \left(1-C\alpha^{-\frac{k}{n}}\right)\sigma_k\leq \kappa_1\sigma_k^{11}\leq \left(1+C\alpha^{-\frac{k}{n}}\right)\sigma_k\leq 2\sigma_k,
\end{align}
by choosing $\alpha$ sufficiently large.

\medskip

{\bf Case 2-1:} There exists $m\leq l<k$ such that $\kappa_l\geq \delta \kappa_1$ and $\kappa_{l+1}\leq \delta^\prime \kappa_1$, where we have chosen $\delta_0=\min\{\frac{2k}{n},\frac{1}{2}\}$ and $\epsilon$ is a small constant to be determined later.

In this case, by Lemma \ref{Concavity-Lemma}, we have
\begin{align}\label{Curvature-17}
-\sum_{p\neq q} F^{pp,qq}h_{pp1}h_{qq1}+\frac{\left( \sum_i F^{ii}h_{ii1}\right)^2}{\sigma_k}\geq (1-\epsilon)\frac{\sigma_k h_{111}^2}{\kappa_1^2}-\delta_0\sum_{i>l}\frac{F^{ii}h_{ii1}^2}{\kappa_1 },
\end{align}

By Lemma \ref{Sigma_k-lemma} and the choice of $\delta_0$, we have
\begin{align}\label{Curvature-18}
\frac{2}{\kappa_1(\kappa_1-\kappa_i)}-\delta_0\frac{1}{\kappa_1^2}=\frac{(2-\delta_0)\kappa_1+\delta_0\kappa_i}{\kappa_1^2(\kappa_1-\kappa_i)}\geq 0.
\end{align}

Plugging (\ref{Curvature-17}) and (\ref{Curvature-18}) into (\ref{Curvature-14}), we have
\begin{align}\label{Curvature-19}
0\geq &\ -\frac{F_1^2}{\kappa_1\sigma_k}+(1-\epsilon)\frac{\sigma_k h_{111}^2}{\kappa_1^3}+\left(\alpha \phi^\prime-1\right) \sum_iF^{ii}+ \frac{n-1}{n}\frac{a}{u-a}\sum_iF^{ii}\kappa_i^2\\\nonumber
&\ +(k-p)F \kappa_1 +p(p-1)\frac{F}{\kappa_1}  \frac{ u_1^2}{u^2}-C\alpha-C.
\end{align}

By the critical equation (\ref{Curvature-critical}) and (\ref{Curvature-16}), together with the fact that $p\leq 1$, we have
\begin{align}\label{Curvature-20}
&\ \frac{n-1}{n}\frac{a}{u-a} F^{11}\kappa_1^2 -\frac{F_1^2}{\kappa_1\sigma_k}+(1-\epsilon)\frac{\sigma_k h_{111}^2}{\kappa_1^3}+p(p-1)\frac{F}{\kappa_1}  \frac{ u_1^2}{u^2}\\\nonumber
=&\ \frac{n-1}{n}\frac{a}{u-a}F^{11}\kappa_1^2-\frac{\left(\varphi\cdot pu^{p-1}u_1+u^p\varphi_1\right)^2}{\kappa_1\sigma_k} +(1-\epsilon)\frac{F}{\kappa_1}\left(\frac{u_1}{u-a}-\alpha\Phi_1\right)^2\\\nonumber
&\ +p(p-1)\frac{F}{\kappa_1}  \frac{ u_1^2}{u^2}\\\nonumber
\geq &\ C\kappa_1-\epsilon \frac{F}{\kappa_1}\frac{u_1^2}{(u-a)^2}-C\alpha-C\\\nonumber
\geq &\ C\kappa_1-C\epsilon \kappa_1-C\alpha-C\\\nonumber
\geq&\  (C-C\epsilon)\kappa_1-C\alpha-C\geq 0,
\end{align}
by choosing $\epsilon$ sufficiently small and assuming $\kappa_1$ sufficiently large.

Plugging (\ref{Curvature-20}) into (\ref{Curvature-19}), we have
\begin{align*}
0\geq (k-p)F \kappa_1 -C\alpha-C.
\end{align*}

It follows that $\kappa_1\leq C$.

\medskip

{\bf Case 2-2:} For each $2\leq l\leq k$, there exists $\delta_l$ such that $\kappa_l\geq \delta_l\kappa_1$.

In this case, by (\ref{Curvature-15-1}), we have
\begin{align*}
\alpha^{-\frac{k-2}{n}}\geq \kappa_2\cdots \kappa_{k-1}\geq \delta_2\cdots\delta_{k-1}\kappa_1^{k-1}.
\end{align*}

It follows that $\kappa_1\leq C$.

The theorem is now proved.

\end{proof}

\begin{rema}
By the proof above, the estimate holds for $p\leq 1+\epsilon_0$ as well, where $\epsilon_0>0$ is a small constant depending only on $n,k,p,\|M\|_{C^1}, \inf \varphi$ and $\|\varphi\|_{C^2}$.
\end{rema}

\end{document}